\documentclass[11pt,a4paper]{article}
\usepackage[latin1]{inputenc}
\usepackage{a4wide}
\usepackage{amsmath}
\usepackage{amsthm}
\usepackage{amssymb}  
\usepackage{amsfonts}
\usepackage{cite}
\newtheorem{thm}{Theorem}
\newtheorem{lem}{Lemma}

\newtheorem{defin}{Definition}
\newtheorem{prop}{Proposition}

\title{Non-integrable supersymmetries and their classification for $\mathfrak{gl}(1,1)$ and $\mathfrak{sl}(1,1)$ }
\author{M. Kalus\footnote{Research supported by the SFB/TR 12, Symmetry and Universality in Mesoscopic Systems, of the Deutsche Forschungsgemeinschaft.}}
\date{}
\begin{document}
\maketitle
\begin{abstract}
\noindent
Infinitesimal supersymmetries over classical Lie groups that do not necessarily integrate to Lie supergroups are described. They yield a notion of supersymmetry that is less rigid than the assumption of a Lie supergroup action but still implies an underlying action of a Lie group. In contrast to Lie supergroups, the arising representation-theoretical Lie supergroups (RTLSG) occur as families associated to  Harish-Chandra superpairs. However morphisms of RTLSGs directly correspond to morphisms of Harish-Chandra superpairs. Particular RTLSGs can be derived from the explicit constructions of Lie supergroups given by Berezin and Kostant. The Lie superalgebras $\mathfrak{gl}(1,1)$ or  $\mathfrak{sl}(1,1)$ appearing also in higher dimensional classical Lie superalgebras,  provide interesting first examples of RTLSGs. A classification of RTLSGs associated to real and complex $\mathfrak{gl}(1,1)$- and  $\mathfrak{sl}(1,1)$-Harish-Chandra superpairs is given by parameter spaces and complete sets of invariants. The underlying Lie group is assumed to be connected but possibly not simply connected. 
\end{abstract}
\setcounter{tocdepth}{2}

\section{Introduction}

\noindent
Advancing bosonic systems with symmetry Lie group $G$ leads to the study of fields contained in the symmetric powers of a representation space of $G$. This representation space is e.g.~the space of sections of a $G$-homogeneous vector bundle.  Extending a bosonic system to fermions, antisymmetric powers of a representation space of $G$ are added.\footnote{In most examples $G$ is close to a product of two groups each factor acting on bosonic and fermionic part respectively and trivial on the complement.} The infinitesimal version of this action is given by the representation of the associated Lie algebra $\mathfrak g_{\bar 0}=Lie(G)$ of left invariant derivations. Here the representation of $\mathfrak g_{\bar 0}$ on the algebra of fields is equivariant with respect to the adjoint action of $G$ on $\mathfrak g_{\bar 0}$. In particular the representation $\mathfrak g_{\bar 0} \to Der(\mathcal C^\infty_G)$, resp. in the complex case $\mathfrak g_{\bar 0} \to Der(\mathcal O_G)$ is equivariant.

\bigskip\noindent
Supersymmetry in the above setting is usually understood as the presence of an action of a Lie supergroup $\mathcal G$ over $G$ associated to a Lie superalgebra $\mathfrak g=\mathfrak g_{\bar 0}\oplus \mathfrak g_{\bar 1}$. A Lie supergroup can be realized as a supermanifold $\mathcal G=(G,\mathcal F)$ with sheaf of superfunctions $\mathcal F=\mathcal C^\infty_G \otimes \Lambda \mathfrak g_{\bar 1}^\ast$, resp. $\mathcal F=\mathcal O_G \otimes \Lambda \mathfrak g_{\bar 1}^\ast$  and morphisms for multiplication and inverse. The infinitesimal version of a Lie supergroup action is a representation of the Lie superalgebra $\mathfrak g$ being again equivariant with respect to the adjoint action of the underlying Lie group $G \subset \mathcal G$. Observing the supergroup itself, the analogue $G$-equivariance of the representation $\rho_{\bar 0}:\mathfrak g_{\bar 0} \to Der_{\bar 0}(\mathcal F)$ holds. Explicitly -- in addition to the representation of $\mathfrak g_{\bar 0}$ on numerical functions mentioned above -- the representation on the odd derivatives is given by $X.\alpha=\alpha([\cdot,X])$ for $X\in \mathfrak g_{\bar 0}$ and $\alpha \in \mathfrak g^\ast_{\bar 1}$. 

\bigskip\noindent
Starting from the symmetric situation, the assumption of supersymmetry is rather rigid by requiring the presence of a Lie supergroup action. Here this assumption is released nevertheless keeping the underlying Lie group structure and allowing infinitesimal generators of odd symmetries -- the later forming a Lie superalgebra $\mathfrak g$ together with the infinitesimal even symmetries. This Lie superalgebra is represented with $\rho:\mathfrak g \to Der(\mathcal F)$ on the sheaf of  superfunctions $\mathcal F$ on the supermanifold $\mathcal G$. Furthermore assumed are the natural assumption of $G$-equivariance of this representation forcing its even part to be $\rho_{\bar 0}$, and the following transitivity condition:  the common kernel of $\rho(\mathfrak g)$ is the constant numerical functions. The arising objects are called representation-theoretical Lie supergroups (RTLSG).

\bigskip\noindent
Higher dimensional Lie superalgebras such as $\mathfrak{gl}(n,m)$, $\mathfrak{sl}(n,m)$ or $\mathfrak{osp}(n,m)$ are constructed from the Lie subsuperalgebras associated to their odd root spaces being exactly those of type $\mathfrak{sl}(1,1)$ and $\mathfrak{osp}(1,2)$ (see \cite{Kac}). So the first step in the analysis of RTLSGs of higher dimension is the classification of $\mathfrak{gl}(1,1)$-, resp. $\mathfrak{sl}(1,1)$-RTLSGs.

\bigskip\noindent
\textbf{Contents.} The second section contains the rigorous definition of RTLSGs and their morphisms. Parameter spaces being non-singular algebraic varieties are given for $\mathfrak{gl}(1,1)$ and $\mathfrak{sl}(1,1)$. In particular a complete set of invariants parametrizing the isomorphy classes of RTLSG structures is listed. $G$-equivariant Batchelor models of Lie supergroups provide special but not all examples of RTLSGs. 

\bigskip\noindent
The parameter space for $\mathfrak{gl}(1,1)$ is determined in the third sections yielding a system of 25 equations in 16 variables. As it turns out the parameter space is a 6-dimensional algebraic variety $\mathcal V_{\mathfrak {gl}}$ with generically 3-dimensional orbits of the automorphism group of the associated Harish-Chandra superpair. 

\bigskip\noindent
In section four the $\mathfrak{sl}(1,1)$-RTLSG structures are extended to $\mathfrak{gl}(1,1)$-RTLSGs yielding the points in a subvariety $\mathcal V_{\mathfrak {sl}}\subset \mathcal V_{\mathfrak {gl}}$. Automorphisms of $\mathfrak{sl}(1,1)$-RTLSGs are also extendable to $\mathfrak{gl}(1,1)$-RTLSGs and guarantee a classification analog to the $\mathfrak{gl}(1,1)$ case. 

\bigskip\noindent
Finally section five concerns RTLSG structures over non-simply connected Harish-Chandra superpairs. The classification for $\mathfrak{sl}(1,1)$-RTLSGs remains unchanged from the simply connected case while the generic $\mathfrak{gl}(1,1)$-case differs. An example for the case of first homotopy $\mathbb Z^2$ of the underlying Lie group is added.

\bigskip\noindent
\textbf{Preliminaries.} The RTLSGs discussed here are defined in the smooth and in the holomorphic case. So the field $\mathbb K$ will denote $\mathbb R$ or $\mathbb C$ and the sheaf of numerical functions on a Lie group $G$ is denoted by $\mathcal C^\infty_G$, resp. $\mathcal O_G$. The underlying Lie groups are assumed to be connected. Provided a $\mathbb Z$-grading on  superfunctions is fixed, the projection onto the $k$-th $\mathbb Z$-degree part followed by evaluation at the neutral element $e$ of the underlying Lie group is denoted by $ev_e^k$. Furthermore $\mathfrak S_k$ denotes the group of permutations of $k$ elements.

\bigskip\noindent
The Lie superalgebra $\mathfrak{gl}(1/1)$ can be represented by $2 \times 2$-matrices 
with entries in $\mathbb K$. The bracket is the supercommutator. For later calculation fix 
\begin{align}\label{mat}
A= \begin{pmatrix} 1 & 0 \\ 0 & 0 \end{pmatrix}, \  B= \begin{pmatrix} 0 & 0 \\ 0 & 1 \end{pmatrix}  \in \mathfrak{gl}(1,1)_{\bar 0},  \quad \mbox{ and } \quad 
C= \begin{pmatrix} 0 & 1 \\ 0 & 0 \end{pmatrix}  D= \begin{pmatrix} 0 & 0 \\ 1 & 0 \end{pmatrix}\  \in \mathfrak{gl}(1,1)_{\bar 1}
\end{align}
The Lie superalgebra $\mathfrak{sl}(1,1)$ is spanned by $A+B$, $C$ and $D$. 
Furthermore let $f_{n,m}$, $n,m \in \mathbb Z$ be the unique numerical functions on the underlying Lie group near the neutral element $e$ that satisfy
\begin{align}\label{eq:306}
&A.f_{n,m}=n \cdot f_{n,m}, \quad B.f_{n,m}=m \cdot f_{n,m} \quad \mbox{ and } \quad  f_{n,m}(e)=1 
\end{align} 
and let $C^\ast$ and $D^\ast$ denote the dual basis to $\{C,D\}$.

\bigskip\noindent
The definition of $\mathfrak{gl}(1,1)$-RTLSG structures and their parametrization in the simply connected case is a part of the PhD thesis of the author (see \cite{Dis}).\footnote{An ``RTLSG`` in the notion of this paper is called ``planed LRT Lie supergroup'' in \cite{Dis}.}

\bigskip\noindent
\textbf{Acknowledgments.} The author thanks A. Huckleberry for bringing up the problem of non-integrable supersymmetries and for many valuable discussions.

\section{Representation-theoretical Lie supergroups}
\noindent
Formalizing the intention stated in the introduction yields the following definition:
\begin{defin}
A representation-theoretical Lie supergroup (RTLSG) is a tuple $(G,\mathfrak g,\mathcal F,\rho)$ where
\begin{itemize}
 \item $(G,\mathfrak g)$ is a Harish-Chandra superpair\footnote{Requires $\mathfrak g_{\bar 0}=Lie(G)$ and integrability of the adjoint representation of $\mathfrak g_{\bar 0}$ on $\mathfrak g_{\bar 1}$.}
 \item $(G,\mathcal F)$ is the trivialized supermanifold with $\mathcal F=\mathcal C^\infty_G \otimes \Lambda \mathfrak g_{\bar 1}^\ast$, resp. $\mathcal O_G \otimes \Lambda \mathfrak g_{\bar 1}^\ast$ 
 \item $\rho:\mathfrak g \to \mbox{Der}(\mathcal F)$ is a Lie superalgebra representation such that $\rho|_{\mathfrak g_{\bar 0}}=\rho_{\bar 0}$ is represented \begin{itemize} \item[(i)] on numerical functions by left-invariant derivations \item[(ii)] on $\alpha \in \mathfrak g_{\bar 1}^\ast$ by $X.\alpha=\alpha([\cdot, X])$ \end{itemize}
\end{itemize}
 satisfying transitivity:
 \begin{align}\label{Trans}
  & \mbox{If }\ f \in \mathcal F \mbox{ with }X.f\equiv 0 \ \ \forall \ X \in \mathfrak g \mbox{ then } f \in \mathbb K .
 \end{align}
A morphism of RTLSGs $(G,\mathfrak g,\mathcal F,\rho)\to (G^\prime,\mathfrak g^\prime,\mathcal F^\prime,\rho^\prime)$ is a morphism of Harish-Chandra superpairs $(\Psi,\psi):(G,\mathfrak g)\to (G^\prime,\mathfrak g^\prime)$ inducing the morphism of supermanifolds $$(\Psi,\psi^\#):(G,\mathcal F) \to (G^\prime,\mathcal F^\prime)\quad  \mbox{ via }\quad \begin{array}{ll} \psi^\#(f):=\Psi^\ast(f) &\mbox{ for } f \in \mathcal C^\infty_{G^\prime} \ (\mbox{resp. }\mathcal O_{G^\prime}) \\ \psi^\#(\alpha):=\alpha \circ \psi|_{\mathfrak g_{\bar 1}} &\mbox{ for } \alpha \in  (\mathfrak g_{\bar 1}^\prime)^\ast \end{array}$$ such that $\rho(X)\circ \psi^\#=\psi^\# \circ \rho^\prime(\psi(X))$ for all $X \in \mathfrak g$.\footnote{This is automatically satisfied for $X \in \mathfrak g_{\bar 0}$.}  Abbreviate $\rho(X)(f)=X.f$.
\end{defin}

\noindent
\textbf{Example 1:} The construction of a Lie supergroup $\mathcal G$ by Kostant (see \cite{Kost}) provides an explicit   identification of superfunctions on the supermanifold $\mathcal G$ with $\mathcal C^\infty_G \otimes \Lambda \mathfrak g_{\bar 1}^\ast$ (see \cite{Kosz}) via dualization of the Lie-Hopf operator superalgebra. The resulting representation of the Lie superalgebra on the superfunctions yields an RTLSG structure. This is shown in the authors dissertation (see \cite{Dis}). The example in the $\mathfrak{gl}(1,1)$ case is given in section \ref{subsec} of this article. Note that different identifications of superfunctions with $\mathcal C^\infty_G \otimes \Lambda \mathfrak g_{\bar 1}^\ast$ may yield different RTLSG structures. For example it is shown in \cite{Dis} that the RTLSG structure induced by Berezin's construction of a Lie supergroup (see \cite{Ber}) via analytic continuation on Grassmann variables is not isomorphic to the first one. 

\bigskip\noindent 
The example may raise the suspicion that any RTLSG structure can be viewed as a Batchelor model $\mathcal G \to (G,\mathcal F)$ of the Lie supergroup associated to the Harish-Chandra superpair $(G,\mathfrak g)$ satisfying the condition that the induced representation of $\mathfrak g$ on $\mathcal F$ restricts to $\rho_0$. Briefly: that RTLSG structures are $G$-equivariant Batchelor models of the Lie supergroup. In the end of section \ref{subsec} this suspicion is ruled out for the case of $\mathfrak{gl}(1,1)$-RTLSG. But it holds for $\mathfrak{sl}(1,1)$ due to the trivial representation of the even part of $\mathfrak{sl}(1,1)$ on the odd part. This is shown in the end of section \ref{secg}.

\bigskip\noindent 
The classification of RTLSG-structures in general depends on the particular Harish-Chandra superpair $(G,\mathfrak g)$ for a fixed Lie superalgebra $\mathfrak g$. In the case of $\mathfrak{gl}(1,1)$ the underlying Lie group is a quotient of $\mathbb K^2$. It hence is of the form $G=\mathbb K^2/\Gamma$ where $\Gamma$ is a lattice of maximal rank $2$, resp. $4$. Note that not all quotient groups $\mathbb K^2/\Gamma$ yield Harish-Chandra superpairs since the lattice may prevent the adjoint representation of $\mathfrak g_{\bar 0}$ on $\mathfrak g_{\bar 1}$ from integration. The quotient also shrinks the group of automorphisms of Harish-Chandra superpairs but it does not interfere with the existence of RTLSG structures as will be shown. In the case of $rk(\Gamma)\leq 1$ an injective parametrization of $\mathfrak{gl}(1,1)$ and $\mathfrak{sl}(1,1)$-RTLSG structures is given. For the remaining cases a convenient set of invariants whose values  distinguish the isomorphy classes and parametrize them up to a discrete group action are provided.
The results  are summarized in the following Theorem:

\begin{thm}\label{thm:res} Let $G=\mathbb K^2/ \Gamma$, resp. $\mathbb K/\Gamma$ be the underlying Lie group of a $\mathfrak{gl}(1,1)$, resp. $\mathfrak{sl}(1,1)$-Harish-Chandra superpair. This classifies the associated RTLSGs:
\begin{itemize}
 \item[(a)] (Case $rk(\Gamma)=0$.) The parameter space of $\mathfrak{gl}(1,1)$-RTLSG structures can be realized as a 6-dimensional algebraic submanifold $\mathcal V_{\mathfrak{gl}}$ of degree $3$ in $\mathbb K^{8}$ including a 5-dimensional submanifold $\mathcal V_{\mathfrak{sl}}$. The orbits of the automorphism group on $\mathcal V_{\mathfrak{gl}}\backslash \mathcal V_{\mathfrak{sl}}$ are generically 3-dimensional and injectively parametrized up to a $\mathfrak S_2$-action by the values of the following invariants:
 \begin{align}\label{inv1}
 \begin{array}{l}
I_1:=ev_e^0\left((C.D^\ast)\cdot (D.C^\ast) \right):\mathcal V_{\mathfrak {gl}}\to \mathbb K \\ 
I_2:=ev_e^0(\left(C.C^\ast \right), ev_e^0\left(D.D^\ast \right)):\mathcal V_{\mathfrak {gl}}\to \mathbb K^2
 \end{array} 
\end{align}
The $\mathfrak S_2$-action is given by interchanging the components of $I_2$.
 \item[(b)]  (Case $rk(\Gamma)=0$.) The parameter space of $\mathfrak{sl}(1,1)$-RTLSG can be identified with $\mathcal V_{\mathfrak{sl}}$ by extension of the RTLSG structure. The automorphism group of $\mathfrak{sl}(1,1)$-RTLSGs yields the same orbit structure as for $\mathfrak{gl}(1,1)$. They are generically 1-dimensional and are injectively parametrized up to a $\mathfrak S_2$-action by the values of the restriction to $\mathcal V_{\mathfrak {sl}}$ of the invariants $I_1$ and $I_2$ in (\ref{inv1}) and of \footnote{In particular the subspaces in the definition of the following map are 1-dimensional and the restricted map is an invariant up to the $\mathfrak S_2$-action on $\mathcal V_{\mathfrak {sl}}$.}
  \begin{align}\label{inv2}
  &I_f:\mathcal V_{\mathfrak {gl}}\to \mathbb P^1(\mathbb K) \nonumber \\
&I_f:=\mbox{span}\left\{\left(\begin{array}{c}ev_e^0(C.C^\ast)\cdot ev_e^2(D^\ast D.f)\\ ev_e^0(D.D^\ast)\cdot ev_e^2(C^\ast C.f) \end{array}\right) , \left(\begin{array}{c}ev_e^0(C.D^\ast)\cdot ev_e^2(C^\ast D.f)\\ ev_e^0(D.C^\ast)\cdot ev_e^2(D^\ast C.f) \end{array}\right) \right\}
\end{align}
with $f=f_{1,0}$. The $\mathfrak S_2$-action is given by interchanging the components of $I_2$ and $I_{f}$ simultaneously.
\item[(c)] (Case $rk(\Gamma) \neq 0$.) All structures in (a) and (b) pass the quotient. We have for the lattice\footnote{The second summand is zero in the real or $\mathfrak{sl}(1,1)$-case.} $\Gamma \subset \mathbb C (A+B)+ \pi i\mathbb Z (A-B)$ and so in particular $rk({\Gamma})\leq \mbox{dim}_\mathbb R(\mathbb K^2)-1$. The classification of structures in $\mathcal V_{\mathfrak{sl}}$ still holds. In $\mathcal V_{\mathfrak{gl}}\backslash \mathcal V_{\mathfrak{sl}}$ assuming $rk(\Gamma) = 1$ with generator $g=\nu (A+B)+ \pi ik (A-B)$ the $\mathfrak{gl}(1,1)$-RTLSG structures are injectively parametrized up to a $\mathfrak S_2$-action by the values of the restriction to $\mathcal V_{\mathfrak{gl}}\backslash \mathcal V_{\mathfrak{sl}}$ of the invariants $I_1$, $I_2$  in (\ref{inv1}) and of $I_{f_g}$ defined as in (\ref{inv2}) with 
\begin{align}\label{inv12}
 f_g=(\nu- \pi ik)f_{1,0}-(\nu+ \pi ik)f_{0,1} \ .
\end{align}
The $\mathfrak S_2$-action is given by interchanging the components of the value of $I_2$ and $I_{f_g}$ simultaneously.
\item[(d)] (Case $rk(\Gamma) \neq 0$.) The isomorphism classes of RTLSG structures in $\mathcal V_{\mathfrak{gl}}\backslash \mathcal V_{\mathfrak{sl}}$ for $rk(\Gamma)\in\{2,3\}$ are injectively parametrized up to a discrete group action by the values of the restriction to $\mathcal V_{\mathfrak{gl}}\backslash \mathcal V_{\mathfrak{sl}}$ of the invariants $I_1$ and $I_2$  in (\ref{inv1}) and $I_{f_{g}}$ and $I_{f_h}$ defined by (\ref{inv2}) and (\ref{inv12}) for two independent $g,h \in \Gamma$.
\end{itemize}
\end{thm}

\noindent
An example for the classification in case \textit{(d)} in the Theorem is given in the end of section \ref{sec1c}. In particular it shows that the acting  discrete group comes from extendable automorphisms of $\Gamma$ and may not be finite.

\bigskip\noindent 
\textbf{Remark.} Consider the {deformation of an RTLSG} $(G,\mathfrak g,\mathcal F,\rho)$ to be a pair $((\mathcal W,q),R)$ of a pointed complex manifold and a smooth map of Fr\'echet manifolds $R:\mathcal W \times \mathfrak g \to \mbox{Der}(\mathcal F)$  such that $(G,\mathfrak g,\mathcal F,R|_{\{z\} \times \mathfrak g})$ is an RTLSG for all $z \in \mathcal W$ and $R|_{\{q\} \times \mathfrak g}=\rho$. Note that Theorem \ref{thm:res} yields a deformation covering all possible RTLSG structures starting e.g.~from a structure induced by the Lie supergroup. It will be shown in the end of section \ref{subsec} for $\mathfrak{gl}(1,1)$ and $\mathfrak{sl}(1,1)$ that this deformation is connected in the complex case and has two components in the real case.

\section{Parametrization of the $\mathfrak{gl}(1,1)$-structures}\label{secs}

\noindent
Classifying RTLSG structures is classifying the possible representations $\rho$ up to isomorphy. This is done for $\mathfrak{gl}(1,1)$ here. 
Until section \ref{sec1c} the underlying Lie group is assumed to be simply connected hence the Harish-Chandra superpair considered is $(\mathbb K^2,\mathfrak{gl}(1,1))$. 
 It is $[\mathfrak g_{\bar 0},\mathfrak g_{\bar 0}]=0$ and with the notation in (\ref{mat}) 
\begin{align}
&[A,C]=-[B,C]=C \quad \mbox{and} \quad -[A,D]=[B,D]=D, \label{eq:303} \\
&[C,D]=A+B \quad \mbox{and} \quad [C,C]=[D,D]=0. \nonumber
\end{align}
The representation of $\mathfrak g_{\bar 0}$ on $\mathfrak g_{\bar 1}^\ast$ is given by
\begin{align}
& -A.C^\ast=B.C^\ast=C^\ast  , \qquad A.D^\ast=-B.D^\ast=D^\ast,  \label{eq:305} 
\end{align}
and finally $A.(C^\ast \wedge D^\ast)= B.(C^\ast \wedge D^\ast)=0$. 
\subsection{The 16 structural Constants}
Using (\ref{eq:303}) and  (\ref{eq:306}) it follows that 
\begin{align*}
&C.f_{n,m}=[A,C].f_{n,m}=(AC-CA).f_{n,m}=A.(C.f_{n,m})-nC.f_{n,m} \\  
& \qquad \Rightarrow \ A.(C.f_{n,m})=(n+1) \cdot C.f_{n,m} \ \ \mbox{and} \\
&C.f_{n,m}=-[B,C].f_{n,m}=-(BC-CB).f_{n,m}=-B.(C.f_{n,m})+mC.f_{n,m} \\ 
&\qquad \Rightarrow \ B.(C.f_{n,m})=(m-1) \cdot C.f_{n,m} \ .
\end{align*} 
Thus $C.f_{n,m}$ lies in the eigen space of $A$ of the eigenvalue $n+1$ and of $B$ for the eigenvalue $m-1$. 
Since $C.f_{n,m}$ is of the form $g_1 C^\ast+g_2 D^\ast$ for $g_1,g_2 \in \mathcal C^\infty_G$, resp. $\mathcal O_G$, 
it follows that $g_1$ is a scalar multiple of $f_{n+2,m-2}$ and $g_2$ of $f_{n,m}$. 
A similar calculation holds for $D.f_{n,m}$. With eight scalar constants 
$c_{C^\ast}^z,c_{C^\ast}^w,c_{D^\ast}^z,c_{D^\ast}^w,d_{C^\ast}^z,d_{C^\ast}^w,d_{D^\ast}^z$ and $d_{D^\ast}^w$ denote 
\begin{align}\label{const1}
&C.f_{1,0}=c_{C^\ast}^z f_{3,-2}C^\ast+c_{D^\ast}^zf_{1,0}D^\ast && \mbox{and} 
&& C.f_{0,1}=c_{C^\ast}^w f_{2,-1}C^\ast+c_{D^\ast}^wf_{0,1}D^\ast \ ,  \\
&D.f_{1,0}=d_{C^\ast}^z f_{1,0}C^\ast+d_{D^\ast}^zf_{-1,2}D^\ast && \mbox{and} \nonumber
&& D.f_{0,1}=d_{C^\ast}^w f_{0,1}C^\ast+d_{D^\ast}^wf_{-2,3}D^\ast\ . 
\end{align}
Now using (\ref{eq:303}) and (\ref{eq:305}) it follows that 
\begin{align*} 
&C.C^\ast=[A,C].C^\ast=(AC-CA).C^\ast=A.(C.C^\ast)+C.C^\ast, \nonumber \\ 
&C.C^\ast=-[B,C].C^\ast=-(BC-CB).C^\ast=-B.(C.C^\ast)+C.C^\ast,  \\
&C.D^\ast=[A,C].D^\ast=(AC-CA).D^\ast=A.(C.D^\ast)-C.D^\ast \qquad \mbox{and} \nonumber \\ 
&C.D^\ast=-[B,C].D^\ast=-(BC-CB).D^\ast=-B.(C.D^\ast)-C.D^\ast \nonumber 
\end{align*}
and an analogue calculation for $D$ leads to 
\begin{align*}
&A.(C.C^\ast)=0, \quad B.(C.C^\ast)=0,\quad A.(C.D^\ast)=2C.D^\ast \ \ \mbox{and} \ \ B.(C.D^\ast)=-2C.D^\ast \\ 
&A.(D.C^\ast)=-2 D.C^\ast, \quad B.(D.C^\ast)=2 D.C^\ast,\quad A.(D.D^\ast)=0 \ \ \mbox{and} \ \ B.(D.D^\ast)=0 \ .  
\end{align*}
So with eight constants 
$c_1^{C^\ast}$,$c_1^{D^\ast}$,
$c_\wedge^{C^\ast}$,$c_\wedge^{D^\ast}$,$d_1^{C^\ast}$,$d_1^{D^\ast}$,$d_\wedge^{C^\ast}$ and $d_\wedge^{D^\ast}$ it is 
\begin{align}\label{const2}
&C.C^\ast=c_1^{C^\ast}+c_\wedge^{C^\ast} C^\ast \wedge D^\ast && \mbox{and} 
&& C.D^\ast=f_{2,-2}(c_1^{D^\ast}+c_\wedge^{D^\ast} C^\ast \wedge D^\ast)\ , \\ \nonumber
&D.C^\ast=f_{-2,2}(d_1^{C^\ast}+d_\wedge^{C^\ast} C^\ast \wedge D^\ast) && \mbox{and} 
&& D.D^\ast=d_1^{D^\ast}+d_\wedge^{D^\ast} C^\ast \wedge D^\ast\ . \nonumber
\end{align}
A representation of $\mathfrak g$ on $\mathcal F$ by superderivations is uniquely determined by these 16 constants.
\subsection{The 25 structural Conditions}\label{secu}
The conditions  which guarantee that a $\mathfrak g$-representation is defined are now determined in terms of the above parameters. From $[C,C]=0$ it follows that 
\begin{align*}
&0=[C,C].f_{1,0}=2C.C.f_{1,0} \ \Leftrightarrow \  
\left\{ \begin{array}{ll} 2c_{C^\ast}^zc_{D^\ast}^w-2c_{C^\ast}^zc_{D^\ast}^z+c_{C^\ast}^z c_\wedge^{C^\ast}+
c_{D^\ast}^z c_\wedge^{D^\ast}=0 & \ (i) \\ 
c_{C^\ast}^zc_1^{C^\ast}+c_{D^\ast}^zc_1^{D^\ast}=0 & \ (ii) \end{array} \right. \\
&0=[C,C].f_{0,1}=2C.C.f_{0,1} \ \Leftrightarrow \  \left\{ \begin{array}{ll} 2c_{C^\ast}^wc_{D^\ast}^w-2c_{C^\ast}^wc_{D^\ast}^z+c_{C^\ast}^wc_\wedge^{C^\ast}+c_{D^\ast}^wc_\wedge^{D^\ast}=0 & \ (iii) \\ 
c_{C^\ast}^wc_1^{C^\ast}+c_{D^\ast}^wc_1^{D^\ast}=0 & \ (iv) \end{array}   \right.
\end{align*}
Furthermore,
\begin{align*}
&0=[C,C].C^\ast=2C.C.C^\ast \ \Leftrightarrow \  \left\{ \begin{array}{l} c_\wedge^{C^\ast}c_1^{C^\ast}=0 \\ c_\wedge^{C^\ast}c_1^{D^\ast}=0 \end{array} \right.  && \begin{array}{l}  (v) \\ (vi) \end{array}  \\
&0=[C,C].D^\ast=2C.C.D^\ast \ \Leftrightarrow \  \left\{ \begin{array}{l} 2c_{C^\ast}^zc_1^{D^\ast}-2c_{C^\ast}^wc_1^{D^\ast}-c_\wedge^{D^\ast}c_1^{D^\ast}=0  \\ 2c_{D^\ast}^zc_1^{D^\ast}-2c_{D^\ast}^wc_1^{D^\ast}+c_\wedge^{D^\ast}c_1^{C^\ast}=0 \end{array}   \right. && \begin{array}{l}  (vii) \\ (viii) \end{array} 
\end{align*}
In addition, from $[D,D]=0$ it follows that 
\begin{align*}
&0=[D,D].f_{1,0}=2D.D.f_{1,0} \ \Leftrightarrow \  \left\{ \begin{array}{ll} 2d_{D^\ast}^zd_{C^\ast}^w-2d_{D^\ast}^zd_{C^\ast}^z+d_{C^\ast}^z d_\wedge^{C^\ast}+d_{D^\ast}^z d_\wedge^{D^\ast}=0 & \ (ix) \\ d_{C^\ast}^zd_1^{C^\ast}+d_{D^\ast}^zd_1^{D^\ast}=0 & \ (x) \end{array} \right. \\
&0=[D,D].f_{0,1}=2D.D.f_{0,1} \ \Leftrightarrow \  \left\{ \begin{array}{ll} 2d_{D^\ast}^wd_{C^\ast}^w-2d_{D^\ast}^wd_{C^\ast}^z+d_{C^\ast}^wd_\wedge^{C^\ast}+d_{D^\ast}^wd_\wedge^{D^\ast}=0 & \ (xi) \\ d_{C^\ast}^wd_1^{C^\ast}+d_{D^\ast}^wd_1^{D^\ast}=0 & \ (xii) \end{array}   \right.
\end{align*}
and
\begin{align*}
&0=[D,D].D^\ast=2D.D.C^\ast \ \Leftrightarrow \  \left\{ \begin{array}{l} 2d_{C^\ast}^wd_1^{C^\ast}-2d_{C^\ast}^zd_1^{C^\ast}-d_\wedge^{C^\ast}d_1^{D^\ast}=0  \\ 2d_{D^\ast}^wd_1^{C^\ast}-2d_{D^\ast}^zd_1^{C^\ast}+d_\wedge^{C^\ast}d_1^{C^\ast}=0 \end{array} \right. && \begin{array}{l}  (xiii) \\ (xiv) \end{array}   \\
&0=[D,D].D^\ast=2D.D.D^\ast \ \Leftrightarrow \  \left\{ \begin{array}{l} d_\wedge^{D^\ast}d_1^{C^\ast}=0 \\ d_\wedge^{D^\ast}d_1^{D^\ast}=0  \end{array}   \right. && \begin{array}{l}  (xv) \\ (xvi) \end{array} 
\end{align*}
From $[C,D]=A+B$ applied to $f_{1,0}$ and $f_{0,1}$ it follows that
\begin{align*}
&\begin{pmatrix}
c_\wedge^{C^\ast} & c_\wedge^{D^\ast}+2c_{C^\ast}^w-4c_{C^\ast}^z & 0 &2c_{C^\ast}^z &0&0&c_{C^\ast}^z & c_{D^\ast}^z \\
0&-2c_{C^\ast}^w&c_\wedge^{C^\ast}&c_\wedge^{D^\ast}-2c_{C^\ast}^z+4c_{C^\ast}^w&0&0&c_{C^\ast}^w&c_{D^\ast}^w \\
c_1^{C^\ast} & c_1^{D^\ast} &0&0&c_{C^\ast}^z&c_{D^\ast}^z&0&0 \\
0&0&c_1^{C^\ast}&c_1^{D^\ast}&c_{C^\ast}^w&c_{D^\ast}^w&0&0
\end{pmatrix} \\ & \cdot
\begin{pmatrix}
d_{C^\ast}^z &d_{D^\ast}^z& d_{C^\ast}^w & d_{D^\ast}^w & d_1^{C^\ast}& d_1^{D^\ast} & d_\wedge^{C^\ast} & d_\wedge^{D^\ast}
\end{pmatrix}^T=
\begin{pmatrix}
0 \\ 0 \\ 1 \\ 1 
\end{pmatrix} \qquad \qquad 
\begin{array}{l}
(xvii) \\ (xviii) \\ (xix) \\ (xx) 
\end{array}
\end{align*}
Applied to $C^\ast$ and $D^\ast$ yields
\begin{align*}
&\begin{pmatrix}
2c_{C^\ast}^w-2c_{C^\ast}^z&-c_\wedge^{C^\ast}&-c_1^{D^\ast}&0 \\
2c_{D^\ast}^w-2c_{D^\ast}^z+c_\wedge^{C^\ast}&0&c_1^{C^\ast}&0 
\end{pmatrix}  \cdot
\begin{pmatrix}
 d_1^{C^\ast}& d_1^{D^\ast} & d_\wedge^{C^\ast} & d_\wedge^{D^\ast}
\end{pmatrix}^T&&=
\begin{pmatrix}
0 \\ 0 
\end{pmatrix} \qquad  &&
\begin{array}{l}
(xxi) \\ (xxii) 
\end{array}\\
&\begin{pmatrix}
0&2d_{C^\ast}^z-2d_{C^\ast}^w-d_\wedge^{D^\ast}&0&-d_1^{D^\ast}\\
d_\wedge^{D^\ast}&2d_{D^\ast}^z-2d_{D^\ast}^w&0&d_1^{C^\ast}
\end{pmatrix}  \cdot
\begin{pmatrix}
 c_1^{C^\ast}& c_1^{D^\ast} & c_\wedge^{C^\ast} & c_\wedge^{D^\ast}
\end{pmatrix}^T&&=
\begin{pmatrix}
0 \\ 0 
\end{pmatrix}  \qquad &&
\begin{array}{l}
(xxiii) \\ (xxiv) 
\end{array}
\end{align*}

\bigskip\noindent
Finally it is necessary to check which condition guarantees (\ref{Trans}). Let $\widetilde{\mathcal F}$ denote the sheaf  of superfunctions which are in the kernel of all even derivations. From (\ref{eq:305}) and (\ref{eq:306}) it follows that
$$
\widetilde{\mathcal F}=\big\{d+ a \cdot f_{1,-1}C^\ast+b \cdot f_{-1,1}D^\ast+c \cdot C^\ast\wedge D^\ast  \ | \ a,b,c,d \in \mathbb K \big\} \ .
$$
Note that 
$$
C.(C^\ast\wedge D^\ast)=c_1^{C^\ast}D^\ast-f_{2,-2}c_1^{D^\ast}C^\ast \qquad \mbox{and} 
\qquad D.(C^\ast\wedge D^\ast)=f_{-2,2}d_1^{C^\ast}D^\ast-d_1^{D^\ast}C^\ast \ . 
$$
So (\ref{Trans}) can be reformulated as
\begin{align*}
&\begin{pmatrix} C.(f_{1,-1}C^\ast) &C.(f_{-1,1}D^\ast)\\
D.(f_{1,-1}C^\ast) &
D.(f_{-1,1}D^\ast)\end{pmatrix} 
\cdot \begin{pmatrix}a\\ b\end{pmatrix}\equiv\begin{pmatrix}0\\ 0\end{pmatrix} \qquad \Rightarrow \quad a=b=0\\
\mbox{ and not } \ \ & c_1^{C^{\ast}}=c_1^{D^{\ast}}=d_1^{C^{\ast}}=d_1^{D^{\ast}}=0
\end{align*}
which is by calculating the determinant equivalent to:
\begin{align*}
 \left. \begin{array}{c}
   (d_{C^\ast}^w-d_{C^\ast}^z+d_\wedge^{D^\ast})\cdot c_1^{C^\ast} +(c_{D^\ast}^w-c_{D^\ast}^z+c_\wedge^{C^\ast}) \cdot d_1^{D^\ast}  \\  \qquad \qquad \qquad  -(d_{D^\ast}^w-d_{D^\ast}^z+d_\wedge^{C^\ast})\cdot c_1^{D^\ast}-(c_{C^\ast}^w-c_{C^\ast}^z+c_\wedge^{D^\ast})\cdot d_1^{C^\ast}\neq 0 \\  \mbox{ or } \qquad det\begin{pmatrix}c_1^{C^\ast}&c_1^{D^\ast}\\ 
d_1^{C^\ast}&d_1^{D^\ast}\end{pmatrix}\neq 0
\end{array}\qquad  \right\} \ \qquad (xxv) 
\end{align*}
\subsection{The possible Structures}\label{subsec}
First the sets of allowed parameters are determined. 
Denote 
$$
M_{C}:=\begin{pmatrix}c_{C^\ast}^z&c_{D^\ast}^z\\c_{C^\ast}^w&c_{D^\ast}^w \end{pmatrix} \quad \mbox{and} \quad M_{D}:=\begin{pmatrix}d_{C^\ast}^z&d_{D^\ast}^z\\d_{C^\ast}^w&d_{D^\ast}^w \end{pmatrix} 
$$ 
From $det(M_{C})\neq 0$ and $det(M_{D})\neq 0$ it follows with $(ii)$,$(iv)$,$(x)$ and $(xii)$ that $c_1^{C^\ast}=c_1^{D^\ast}=d_1^{C^\ast}=d_1^{D^\ast}=0$ contradicting  
to the equations $(xix)$ and $(xx)$ which also excludes $M_C=M_D=0$. So:
\begin{lem}\label{lem:h15}
$det(M_C)=det(M_D)=0$.
\end{lem}
\begin{proof} A contradiction to $det(M_C)\neq 0$ is constructed.  ($det(M_D)=0$ follows analogously.)
For this note that $(ii)$ and $(iv)$ yield $c_1^{C^\ast}=c_1^{D^\ast}=0$ satisfying $(v)$ to $(viii)$. 
Furthermore, $d_1^{C^\ast}=d_1^{D^\ast}=0$ contradicts $(xix)$ and $(xx)$.

\noindent
\textit{Case 1:} Assume $det(M_{C})\neq 0$ and $d_1^{C^\ast},d_1^{D^\ast}\neq 0$ and set 
$\alpha \in \mathbb K^\times$ with $d_1^{C^\ast}=\alpha d_1^{D^\ast}$.
$(xix)$ and $(xx)$ yield $\alpha c_{C^\ast}^z+c_{D^\ast}^z=\alpha c_{C^\ast}^w+c_{D^\ast}^w=(d_1^{D^\ast})^{-1}$. 
The condition $(xxii)$ implies $c_\wedge^{C^\ast}=2(c_{D^\ast}^z-c_{D^\ast}^w)$ and $(xxiii)$ leads to $c_\wedge^{D^\ast}=0$. Now none of the inequalities $(xxv)$ can be satisfied. 

\noindent
\textit{Case 2:} Assume $det(M_{C})\neq 0$ and $d_1^{C^\ast}\neq 0$, $d_1^{D^\ast}= 0$. 
It is obtained from  $(xxiv)$ that  $c_\wedge^{D^\ast}=0$. $(xix)$ and $(xx)$ yield 
$c_{C^\ast}^z=c_{C^\ast}^w$. Now none of the inequalities $(xxv)$ can be satisfied. 

\noindent
\textit{Case 3:} Assume $det(M_{C})\neq 0$ and $d_1^{C^\ast}=0$, $d_1^{D^\ast}\neq 0$. 
It is obtained from  $(xxi)$ that $c_\wedge^{C^\ast}=0$. $(xix)$ and $(xx)$ yield 
$c_{D^\ast}^z=c_{D^\ast}^w$. Now none of the inequalities $(xxv)$ can be satisfied. 
\end{proof}
\noindent
The following lemmas express $c_\wedge^{C^\ast}$,$c_\wedge^{D^\ast}$,$d_\wedge^{C^\ast}$ and $d_\wedge^{D^\ast}$ in terms of the remaining parameters.

\begin{lem}\label{lem:h1}
$c_\wedge^{C^\ast}=d_\wedge^{D^\ast}=0$. 
\end{lem}
\begin{proof} Assuming $c_\wedge^{C^\ast}\neq0$ a contradiction is derived. ($d_\wedge^{D^\ast}=0$ follows analogously.)
From $c_\wedge^{C^\ast}\neq 0$ it follows from equations $(v)$ and $(vi)$ that $c_1^{C^\ast}=c_1^{D^\ast}=0$. Further, 
$(xix)$ and $(xx)$ imply that $(d_1^{C^\ast},d_1^{D^\ast}) \neq (0,0)$.  Then from $(xxiii)$ and $(xxiv)$ 
if follows that $c_\wedge^{D^\ast}=0$.

\noindent
\textit{Case 1:} Assume that $d_1^{C^\ast}=0$ hence $d_1^{D^\ast}\neq 0$. This contradicts $(xxi)$.

\noindent
\textit{Case 2:} Assume that $d_1^{C^\ast}\neq 0$.  $(xxi)$ and $(xxii)$ yield $c_\wedge^{C^\ast}d_1^{D^\ast}=2(c_{C^\ast}^w-c_{C^\ast}^z)d_1^{C^\ast} $ and $c_\wedge^{C^\ast}=2(c_{D^\ast}^z-c_{D^\ast}^w)$ contradicting $(xxv)$.
\end{proof}
\begin{lem}\label{lem:h2}
$c_\wedge^{D^\ast}=2(c_{C^\ast}^z-c_{C^\ast}^w)$ and $d_\wedge^{C^\ast}=2(d_{D^\ast}^z-d_{D^\ast}^w)$.
\end{lem}
\begin{proof} Assuming $c_\wedge^{{D^\ast}}\neq 2(c_{C^\ast}^z-c_{C^\ast}^w)$ a contradiction is derived. 
($d_\wedge^{C^\ast}=2(d_{D^\ast}^z-d_{D^\ast}^w)$ follows analogously.)
From $(vii)$ it follows that $c_1^{D^\ast}=0$ and from $(ii)$, $(iv)$ and $(viii)$ that 
$c_{C^\ast}^zc_1^{C^\ast}=c_{C^\ast}^wc_1^{C^\ast}=c_\wedge^{D^\ast}c_1^{C^\ast}=0$. Thus $c_1^{C^\ast}\neq 0$ 
contradicts the assumption. Hence $c_1^{C^\ast}=0$. By $(xix)$ and $(xx)$ it is $(d_1^{C^\ast},d_1^{D^\ast})\neq (0,0)$. 
Hence, by $(xxiii)$ and $(xxiv)$ it follows that $c_\wedge^{D^\ast}=0$. Furthermore with Lemma \ref{lem:h1} $(xxi)$ yields $d_1^{C^\ast}=0$, 
since otherwise $c_{C^\ast}^z=c_{C^\ast}^w$ contradicting the assumption. Now $(xix)$ and $(xx)$ yield  $c_{D^\ast}^z=c_{D^\ast}^w$ contradicting $(xxv)$ with Lemma \ref{lem:h1}.
\end{proof}
\noindent
 This yields with a (non-degenerate) coordinate change:

\begin{prop}\label{prop:gl}
The RTLSG-structures for $\mathfrak{gl}(1,1)$ are parametrized by eight parameters $c_1^{C^\ast}$, $c_1^{D^\ast}$, $d_1^{C^\ast}$, $d_1^{D^\ast}$, $c^z$, $c^w$, $d^z$ and $d^w$ with 
$$
M_{C}=\begin{pmatrix}-c_1^{D^\ast}c^z&c_1^{C^\ast}c^z\\-c_1^{D^\ast}c^w&c^{C^\ast}_1c^w \end{pmatrix} \quad \mbox{and} \quad M_{D}=\begin{pmatrix}-d_1^{D^\ast}d^z&d_1^{C^\ast}d^z\\-d_1^{D^\ast}d^w&d^{C^\ast}_1d^w \end{pmatrix} 
$$ 
satisfying
\begin{align}\label{cond}
det\begin{pmatrix} c_1^{C^\ast}& c_1^{D^\ast}\\d_1^{C^\ast}& d_1^{D^\ast}\end{pmatrix}(d^z-c^z)=det\begin{pmatrix} c_1^{C^\ast}& c_1^{D^\ast}\\d_1^{C^\ast}& d_1^{D^\ast}\end{pmatrix}(d^w-c^w)=1
\end{align} 
\end{prop}
\begin{proof} 
Since $det(M_C)=det(M_D)=0$, by Lemma \ref{lem:h15} new parameters $\mu_C$,$\mu_D$, $\nu_C$, $\nu_D$, $c^z$, $c^w$, $d^z$ and $d^w$ can be introduced with
\begin{align*}
 &M_C=\left(\begin{smallmatrix} \mu_C c^z & \nu_C c^z \\ \mu_C c^w & \nu_C c^w \end{smallmatrix} \right)
 \quad  M_D=\left(\begin{smallmatrix} \mu_D d^z & \nu_D d^z \\ \mu_D d^w & \nu_D d^w \end{smallmatrix} \right)
\end{align*}
Equations $(ii)$, $(iv)$, $(x)$ and $(xii)$ yield with new parameters $c_1$ and $d_1$
$$  
c_1^{C^\ast}= \nu_C c_1 \qquad
c_1^{D^\ast}=-\mu_C c_1\qquad
d_1^{C^\ast}=\nu_D d_1\qquad
d_1^{D^\ast}= -\mu_D d_1 
$$
while $(xix)$ and $(xx)$ become:
\begin{align*}
det\left(\begin{smallmatrix} \nu_C & \mu_C \\ \nu_D & \mu_D \end{smallmatrix}\right) (d^zc_1-c^zd_1)=det\left(\begin{smallmatrix} \nu_C & \mu_C \\ \nu_D & \mu_D \end{smallmatrix}\right) (d^wc_1-c^wd_1)=1
\end{align*}
Putting the information of Lemmas \ref{lem:h1} and \ref{lem:h2} into the first inequality of $(xxv)$ contradicts the difference of $(xix)$ and $(xx)$. Hence the second inequality of $(xxv)$ must be satisfied. It yields $c_1\neq 0$, $d_1\neq 0$. Under these conditions the remaining equations are satisfied. The two degrees of freedom gained by the reparametrization of $M_C$ and $M_D$ can be used to normalize $c_1=d_1=1$. A direct calculation of the Jacobian of the coordinate change shows that it is non-degenerate.
\end{proof}

\noindent
 The appearing six dimensional variety in the Proposition will be called $\mathcal V_{\mathfrak{gl}}$. Set
$$\mathcal V_{\mathfrak{sl}}:=\{ \mbox{Parameter sets in } \mathcal V_{\mathfrak{gl}} \mbox{ with } c^z=c^w \mbox{ and } d^z=d^w\}$$
being a 5-dimensional subvariety. 
These are the parameter spaces in Theorem \ref{thm:res}.

\bigskip \noindent
Note that for fixed $w \in \mathbb K^\times$ the irreducible affine variety $ c_1^{C^\ast}d_1^{D^\ast}-c_1^{D^\ast}d_1^{C^\ast}=w $ in $\mathbb K^4$ has one connected component. Since the remaining equations $(d^z-c^z)=(d^w-c^w)=w^{-1}$ in (\ref{cond}) describe affine planes in $\mathbb K^4$ it follows that $\mathcal V_{\mathfrak gl}$ and $\mathcal V_{\mathfrak sl}$ are connected in the complex case. In the real case they decompose in two connected components with respect to the sign of $w$.

\bigskip \noindent
\textbf{Example 1 (continuation).} For the Lie supergroup associated to a $\mathfrak {gl}(1,1)$-Harish-Chandra superpair  the associated RTLSGs given by the construction of Kostant and Koszul (see \cite{Kost} and \cite{Kosz}) and Berezin (see \cite{Ber}) yield the parameter sets 
\begin{align}\label{KK}
\begin{pmatrix} c_1^{C^\ast}& c_1^{D^\ast}\\d_1^{C^\ast}& d_1^{D^\ast}\end{pmatrix}=\begin{pmatrix} -1&0\\0&-1\end{pmatrix} \qquad \mbox{ and } \qquad -c^z=-c^w=d^z=d^w=\frac{1}{2} 
\end{align}
respectively
\begin{align}\label{Ber}
\begin{pmatrix} c_1^{C^\ast}& c_1^{D^\ast}\\d_1^{C^\ast}& d_1^{D^\ast}\end{pmatrix}=\begin{pmatrix} 1&0\\0&1\end{pmatrix} \qquad \mbox{ and } \qquad -c^z=d^w=1, \quad c^w=d^z=0
\end{align}
A detailed calculation can be found in \cite{Dis}. Varying the identification of $\mathcal F$ and $\mathcal C^\infty_G\otimes \Lambda \mathfrak g_1^\ast$, resp. $\mathcal O_G\otimes \Lambda \mathfrak g_1^\ast$ but keeping the representation of $\mathfrak g_0$ fixed allows a super coordinate change of the form 
$$f_{1,0}\mapsto f_{1,0}(1+t_z C^\ast\wedge D^\ast),\quad f_{0,1}\mapsto f_{0,1}(1+t_w C^\ast\wedge D^\ast),\quad C^\ast \mapsto t_{C^\ast} C^\ast, \quad D^\ast \mapsto t_{D^\ast} D^\ast$$
with $t_z,t_w \in \mathbb K$ and $t_{C^\ast}, t_{D^\ast}\in \mathbb K^\times$. 
Starting with (\ref{Ber}), parameters are changed to
\begin{align*}
\begin{pmatrix} c_1^{C^\ast}& c_1^{D^\ast}\\d_1^{C^\ast}& d_1^{D^\ast}\end{pmatrix}=\begin{pmatrix} t_{C^\ast}&0\\0&t_{D^\ast}\end{pmatrix} \qquad \mbox{ and } \qquad \begin{array}{ll} c^z=-t_{C^\ast}^{-1}t_{D^\ast}^{-1}+t_z, \quad & c^w=t_w\\   d^z=t_z \quad & d^w=t_{C^\ast}^{-1}t_{D^\ast}^{-1}+t_w \end{array}
\end{align*}
This characterizes all those RTLSG structures for $\mathfrak{gl}(1,1)$ that can be obtained from $G$-equivariant Batchelor models. Setting $t_{C^\ast}=t_{D^\ast}=-1$ and $t_z=-t_w=\frac{1}{2}$ leads to (\ref{KK}). Obviously not all RTLSGs of $\mathfrak{gl}(1,1)$ can be reached this way. 

\subsection{Isomorphism classes}\label{seciso}
\noindent
It is studied how transporting the representation $\rho$ with isomorphisms of a Harish-Chandra superpair affects the parameters. First the automorphism of the Lie superalgebra $\mathfrak{gl}(1,1)$ have to be determined.
\begin{lem}
The elements of $Aut(\mathfrak{gl}(1,1))$ are maps $\psi_{x,y,u,v}^\pm$ for $u,v \in \mathbb K$, $x,y \in \mathbb K^\times$ with $xy=u+v$ where
 \begin{align*}
&\begin{array}{ll} \psi_{x,y,u,v}^+(A)=uA+(u-1)B \qquad \qquad &\psi_{x,y,u,v}^+(C)=xC  \\ \psi_{x,y,u,v}^+(B)=vA+(v+1)B& \psi_{x,y,u,v}^+(D)= yD \ . \end{array}
\end{align*}
and $\psi_{x,y,u,v}^-=\gamma\circ \psi_{x,y,u,v}^+$ where $\gamma$ is the map interchanging $A \leftrightarrow B$ and $C \leftrightarrow D$.
\end{lem}
\begin{proof}  From $[C,C]=[D,D]=0$ it follows that  $\psi \in Aut(\mathfrak g)$ can either scale $C$ and $D$ by factors $x$, resp $y \in \mathbb K^\times$ or scale and interchange them. By $[C,D]=A+B$ a scaling of $A+B$ with $xy$ is induced. The brackets of $A$ and $B$ with $C$ and $D$ yield $\psi|_{\mathfrak g_{\bar 0}}$.
\end{proof}
\noindent
Setting $b:=u-v-1$ the following result is obtained as a byproduct:
\begin{prop}
The group of automorphisms of $\mathfrak{gl}(1,1)$ is $\mathfrak S_2 \ltimes\left((\mathbb K^\times)^2 \ltimes \mathbb K\right)$ with 
 $$(\sigma,x,y,b) \circ (\sigma^\prime,x^\prime,y^\prime,b^\prime)=(\sigma\sigma^\prime,\sigma^\prime_1(x,y)x^\prime,\sigma^\prime_2(x,y)y^\prime,\mbox{sgn}(\sigma^\prime)x^\prime y^\prime b +b^\prime)$$
\end{prop}
\begin{proof}
 An automorphism $\psi^+_{x,y,u,v}$ is given by the parameters $x$, $y$ and $b$. By direct calculation $(x,y,b)=(x,y,0)\circ (1,1,b)=(1,1,(xy)^{-1}b) \circ (x,y,0)$ while $(x,y,0)$, resp. $(1,1,b)$ commute with morphisms of the same type respectively. Concerning $\gamma$ observe $(x,y,0)\circ \gamma=\gamma\circ (y,x,0)$ and $(1,1,b)\circ \gamma=\gamma\circ(1,1,-b)$.
\end{proof}
\noindent 
An automorphism $\psi_{x,y,u,v}^+$ applied to a parameter set of RTLSG structures for $\mathfrak{gl}(1,1)$ yields:
\begin{align}\label{pauto}
\begin{array}{rlcl}
 \left(\begin{array}{cc} c^z &d^z \\c^w & d^w \end{array}\right)  \ \ &\longmapsto&\ A_{xy,b} \cdot\left(\begin{array}{cc} c^z &d^z \\c^w & d^w  \end{array}\right)\\ \quad \mbox{ and }  \quad \left(\begin{array}{cc}  c_1^{C^\ast}& c_1^{D^\ast}\\d_1^{C^\ast}& d_1^{D^\ast} \end{array}\right) \ &\longmapsto&\ \left(\begin{array}{cc}  c_1^{C^\ast}& \frac{x}{y}c_1^{D^\ast}\\\frac{y}{x}d_1^{C^\ast}& d_1^{D^\ast} \end{array}\right)
  \end{array}
\end{align}
\noindent
where $A_{u,v}=\left(\begin{smallmatrix} 1+v&-v\\1-u&u  \end{smallmatrix}\right)$. The automorphism $\gamma$ rotates both matrices by $180^\circ$ multiplying the first with $-1$.
It follows: 
\begin{lem}\label{342} Set $\alpha:=c^z-c^w$ (which equals $d^z-d^w$ by (\ref{cond})) and regard the $\mathbb K^\times$-action by 
\begin{align*}
 t.(c_1^{D^\ast},d_1^{C^\ast})=(tc_1^{D^\ast},\ t^{-1}d_1^{C^\ast}).
\end{align*}
In the case $\alpha=0$ the isomorphism classes of $\mathfrak{gl}(1,1)$-RTLSG correspond up to the $\mathfrak S_2$-action  to the $\mathbb K^\times$-orbits in the five dimensional subvariety  $\mathcal V_{\mathfrak {sl}}\subset\mathcal V_{\mathfrak {gl}}$.\\
In the case $\alpha\neq 0$ the isomorphism classes of $\mathfrak{gl}(1,1)$-RTLSG correspond up to the $\mathfrak S_2$-action  to the $\mathbb K^\times$-orbits in the four dimensional subvariety of $\mathcal V_{\mathfrak {gl}}$ given by $c^z=d^w=c^w+d^z=0$.
\end{lem}
\begin{proof}
In the first case the matrix $A_{u,v}$ acts trivially. In the second case the arising condition on $u$ and $v$ for moving $c^z$ and $d^w$ to zero can be solved uniquely by $v=-c^z\alpha^{-1}$ and $u=d^z\alpha^{-1}$. 
Note that $u+v\neq0$ follows from $d^z-c^z\neq 0$ in (\ref{cond}).
\end{proof}
\noindent
Now the first part of Theorem \ref{thm:res} can be proved.

\bigskip\noindent 
\textit{Proof of Theorem \ref{thm:res} (a):} The condition (\ref{cond}) defines  non-singular subvarieties $\mathcal V_{\mathfrak {gl}}$ and $\mathcal V_{\mathfrak {sl}}$ by direct calculation of the rank of the Jacobian. The functions in (\ref{inv1}) are invariant under automorphisms in (\ref{pauto}). The $\mathfrak S_2$-action by $\gamma$ interchanges the components of $I_2$. It follows from Lemma \ref{342} ($\alpha \neq 0$) and (\ref{cond}) that the values of the invariants distinguish the orbits. \hfill $\Box$

\section{Parametrization of the $\mathfrak{sl} (1,1)$-structures}\label{secg}
\noindent 
An $\mathfrak{sl}(1,1)$-RTLSG structure and its automorphisms can be continued to the $\mathfrak{gl}(1,1)$-case. This is done here explicitly. The  Lie subsuperalgebra $\mathfrak{sl}(1,1)\subset \mathfrak{gl}(1,1)$ is generated by $H^+:=A+B$, $C$ and $D$.
Note  that due to $[H^+,C]=[H^+,D]=0$  it is $H^+.C^\ast=H^+.D^\ast=0$. Note further that  $f_{1,0}$ and $f_{0,1}$ satisfy $H^+.f_{1,0}=H^-.f_{1,0}=f_{1,0}$ for $H^-:=A-B$ and $H^+.f_{0,1}=-H^-.f_{0,1}=f_{0,1}$ with $f_{1,0}(e)=f_{0,1}(e)=1$. On the underlying simply connected Lie group $G_{\mathfrak {sl}}$ of the $\mathfrak{sl}(1,1)$-RTLSG it is $f:=f_{1,0}|_{G_{\mathfrak {sl}}}=f_{0,1}|_{G_{\mathfrak {sl}}}$.

\bigskip\noindent 
Starting with an $\mathfrak{sl}(1,1)$-RTLSG the relations $[H^+,C]=[H^+,D]=0$ applied to $f$  yield that $C.f$ and $D.f$ are linear combinations of $fC^\ast$ and $fD^\ast$:
$$ C.f=c_{C^\ast}fC^\ast+c_{D^\ast}fD^\ast \qquad \mbox{and} \qquad D.f=d_{C^\ast}fC^\ast+d_{D^\ast}fD^\ast $$
The same is obtained by restricting (\ref{const1}) to $G_{\mathfrak {sl}}$ and fixing 
$$c_{C^\ast}:=c_{C^\ast}^z=c_{C^\ast}^w,\quad  c_{D^\ast}:=c_{D^\ast}^z=c_{D^\ast}^w,\quad d_{C^\ast}:=d_{C^\ast}^z=d_{C^\ast}^w \mbox{ and } d_{D^\ast}:=d_{D^\ast}^z=d_{D^\ast}^w. $$
Furthermore application of $[H^+,C]=[H^+,D]=0$ to $C^\ast$ and $D^\ast$ yields that $C.C^\ast$, $C.D^\ast$, $D.C^\ast$ and $D.D^\ast$ are linear combinations of $1$ and $C^\ast\wedge D^\ast$ coinciding with (\ref{const2}) restricted to $G_\mathfrak {sl}$. So altogether this yields parameters for a $\mathfrak{gl}(1,1)$-RTLSG from an $\mathfrak{sl}(1,1)$-RTLSG.
\begin{lem}
 The parameters defined above yield an embedding of an $\mathfrak{sl}(1,1)$-RTLSG into a $\mathfrak{gl}(1,1)$-RTLSG. 
\end{lem}
\begin{proof}
 It is necessary to check if the conditions in  section \ref{secu} are satisfied on the $\mathfrak{gl}(1,1)$-RTLSG parameter set constructed from an $\mathfrak{sl}(1,1)$-RTLSG. Since equations $(i)$ to $(xxiv)$ come from brackets that are also true in $\mathfrak{sl}(1,1)$ these equations hold. For an $\mathfrak{sl}(1,1)$-RTLSG  by a calculation analogous to the $\mathfrak{gl}(1,1)$-RTLSG case in section \ref{secu} it follows instead of $(xxv)$ that:
\begin{align*}
   d_\wedge^{D^\ast}\cdot c_1^{C^\ast} +c_\wedge^{C^\ast} \cdot d_1^{D^\ast}   -d_\wedge^{C^\ast}\cdot c_1^{D^\ast}-c_\wedge^{D^\ast}\cdot d_1^{C^\ast}\neq 0  \qquad  \mbox{ or } \qquad det\begin{pmatrix}c_1^{C^\ast}&c_1^{D^\ast}\\ 
d_1^{C^\ast}&d_1^{D^\ast}\end{pmatrix}\neq 0
\end{align*}
satisfying $(xxv)$. So the induced parameters satisfy all 25 equations.
\end{proof}
\noindent
 Hence it is possible to canonically continue an $\mathfrak{sl}(1,1)$-RTLSG to a $\mathfrak{gl}(1,1)$-RTLSG in $\mathcal V_{\mathfrak{sl}}$. So setting $c:=c^z=c^w$ and $d:=d^z=d^w$ is follows from Proposition \ref{prop:gl}:
\begin{prop}\label{prop:sl}
The RTLSG-structures for $\mathfrak{sl}(1,1)$ are parametrized by six parameters $c_1^{C^\ast}$, $c_1^{D^\ast}$, $d_1^{C^\ast}$, $d_1^{D^\ast}$, $c$ and $d$ 
satisfying
\begin{align}\label{q1}
det\begin{pmatrix} c_1^{C^\ast}& c_1^{D^\ast}\\d_1^{C^\ast}& d_1^{D^\ast}\end{pmatrix}(d-c)=1
\end{align}
\end{prop}
\noindent
Now the automorphisms of $\mathfrak{sl}(1,1)$ are compared with those of $\mathfrak{gl}(1,1)$.
An automorphism of $\mathfrak{sl}(1,1)$ is induced by an element  $A \in GL(\mathfrak g_{\bar 1})$. The bracket $[C,C]=[D,D]=0$ yields $A(C)=xC$ and $A(D)=yD$ or $A(C)=yD$ and $A(D)=xC$ with $x,y \in \mathbb K^\ast$. Hence the automorphism can be continued to $\mathfrak{gl}(1,1)$ by any $\psi^+_{x,y,u,v}$, resp. $\psi^-_{x,y,u,v}$ where $u+v=xy$. For any allowed $u$, $v$ the automorphisms restrict to the same automorphism on $\mathfrak{sl}(1,1)$.  Now the second part of Theorem \ref{thm:res} can be proved.

\bigskip\noindent 
\textit{Proof of Theorem \ref{thm:res} (b):} The map $I_f$ takes the value $(c_1^{C^\ast}d_1^{D^\ast}\mathbb K +c_1^{D^\ast}d_1^{C^\ast}\mathbb K )\cdot\left(\begin{smallmatrix}d\\c\end{smallmatrix}\right)$. Due to (\ref{q1}) not both $c_1^{C^\ast}d_1^{D^\ast}$ and $c_1^{D^\ast}d_1^{C^\ast}$ can vanish. Furthermore  $I_f$ is invariant under the action (\ref{pauto}) and its values by Lemma \ref{342} (case $\alpha = 0$) and (\ref{q1}) yield the last necessary information for the parametrization of the isomorphy classes. The $\mathfrak S_2$-action by $\gamma$  interchanges the components of $I_2$ and $I_f$ simultaneously.  \hfill $\Box$

\bigskip\noindent 
\textbf{Example 1 (continuation).} Due to $[\mathfrak g_0,\mathfrak g_1]=0$ in $\mathfrak{sl}(1,1)$ a $G$-equivariant supercoordinate change of an RTLSG induced by a Lie supergroup is here given by 
$$ \left(\begin{array}{c} C^\ast\\D^\ast\end{array}\right) \mapsto A \cdot \left(\begin{array}{c} C^\ast\\D^\ast\end{array}\right)  , \quad  f\mapsto f\cdot (1+a \cdot  C^\ast\wedge D^\ast)$$
for $A \in GL_2(\mathbb K)$ and $a \in \mathbb K$. 
This induces a change of parameters by
\begin{align*}
\begin{pmatrix} c_1^{C^\ast}& c_1^{D^\ast}\\d_1^{C^\ast}& d_1^{D^\ast}\end{pmatrix}\mapsto \begin{pmatrix} c_1^{C^\ast} & c_1^{D^\ast}\\d_1^{C^\ast}& d_1^{D^\ast}\end{pmatrix} \cdot A^T, \quad c\mapsto det(A^{-1})(c+a), \quad d\mapsto det(A^{-1})(d+a)
\end{align*}
Starting from one RTLSG this covers by Proposition \ref{prop:sl} all RTLSG structures of $\mathfrak{sl}(1,1)$. So in this case all RTLSG can be realized as $G$-equivariant Batchelor models of the Lie supergroup.

\section{RTLSGs over non-simply connected Lie groups}\label{sec1c}

\bigskip\noindent 
In the non-simply connected case the underlying Lie group of a $\mathfrak{gl}(1,1)$-Harish-Chandra superpair is of the form $G=\mathbb K^2/ \Gamma$ for a lattice $\Gamma$. We fix the basis $C,D$ for $\mathfrak g_{\bar 1}$.   
For any $\xi \in \mathfrak g_{\bar 0}$ with $\exp(\xi)\in \Gamma$ the representation of $\xi$ on $\mathfrak g_{\bar 1}$ by $X\mapsto[X,\cdot]$ has to lie in $\mathfrak D=\{diag(\lambda_1,\lambda_2), \ \lambda_j \in 2 \pi i \mathbb Z\}$ being trivial in the real case. By (\ref{eq:303}) this forces $\xi$ to be of the form:
\begin{align}\label{eq:xi}
 \nu (A+B)+ \pi  ik(A-B)\quad \mbox{for } \nu \in \mathbb K \mbox{ and }   k\in \mathbb Z \quad \mbox{  (}k=0\mbox{ in the real case)}
\end{align}
 A direct calculation yields that an automorphism acts on $\xi$ of the form (\ref{eq:xi})  by:
\begin{align}\label{morph}
\psi_{x,y,u,v}^\pm(\nu (A+B)+ \pi ik(A-B))=(\nu (u+v)+ \pi ik(u-v-1))(A+B)\pm  \pi ik(A-B) \ 
\end{align}
It follows:
\begin{lem}\label{lem:5}
 All $\mathfrak{gl}(1,1)$-, resp. $\mathfrak{sl}(1,1)$-RTLSGs in the simply connected case pass the quotient as well as the actions of $\mathfrak S_2$ and $\mathbb K^\times$ in Lemma \ref{342}.
\end{lem}
\begin{proof}
For an RTLSG in the simply connected case it is by (\ref{const1}) and (\ref{const2}) necessary and sufficient for passing the quotient that superfunctions of the type $f_{2,-2}$ and $f_{-2,2}$ exist globally. The good candidate of these functions on $\mathfrak g_{\bar 0}$-level is $\pm(2A^\ast-2B^\ast)$ mapping any $\xi$ of the form (\ref{eq:xi}) to  an integer multiple of $4\pi i$. Hence  $\exp(\pm(2A^\ast-2B^\ast))$ yield the wanted global functions. 
The symmetries $\pm Id$ of the lattice $\Gamma$ can be extended by (\ref{morph}) to the RTLSG automorphisms of the form $\psi_{x, x^{-1},1,0}^+$, resp. $\psi_{x,- x^{-1},0,-1}^-$ with $x \in \mathbb K^\times$. So both actions pass the quotient.
\end{proof}

\noindent 
Now it is possible to prove the remaining parts of Theorem \ref{thm:res}.

\bigskip\noindent 
\textit{Proof of Theorem \ref{thm:res} (c):} From Lemma \ref{lem:5} it follows that all RTLSGs previously described pass the quotient.  The form of $\xi$ in (\ref{eq:xi}) forces $\Gamma$ to be a lattice of maximal rank $1$ in the real and $3$ in the complex case. Together with Lemma \ref{342} equation (\ref{eq:xi}) yields that the classification stays correct on $\mathcal V_{\mathfrak{sl}}$. The condition on $u$ and $v$ for a map $\psi_{x, y,u,v}^+$ to restrict to the identity on a lattice generated by $\nu(A+B)+ \pi ik(A-B)$ is given via (\ref{morph}) by:
\begin{align}\label{eq:144}
 u(\nu+ \pi ik)+v(\nu- \pi ik)=\nu+ \pi ik
\end{align}
Hence this yields a 1-parameter group action in addition to the actions in Lemma \ref{lem:5}. The map $I_{f_g}$ yields the subspace spanned by $\left(\begin{smallmatrix}(\nu-\pi i k) d^z-(\nu+\pi i k) d^w\\(\nu-\pi i k) c^z-(\nu+\pi i k) c^w\end{smallmatrix}\right)$. The vector does not vanish since vanishing of one entry requires due to $\alpha\neq 0$ that $k\neq 0$. But via (\ref{cond}) vanishing of the vector needs $\alpha(\nu+ik)=\alpha(\nu-ik)$ requiring $k=0$. By direct calculation both entries are invariant under (\ref{pauto}) while $\psi_{x,- x^{-1},0,-1}^-$ interchanges the entries and multiplies with $-1$.
\hfill $\Box$

\bigskip\noindent 
\textit{Proof of Theorem \ref{thm:res} (d):} Note that two linearly independent elements of a lattice yield a system of equations analogue to (\ref{eq:144}) that has only one solution. Hence up to the actions in Lemma \ref{lem:5} the only acting automorphisms are those of the form (\ref{morph}) coming from lattice symmetries. This yields a discrete group action. Since $g$ and $h$ define two independent linear functions on $\mathbb C^2$, the induced $I_{f_g}$ and  $I_{f_h}$ distinguish the orbits.\hfill $\Box$

\bigskip\noindent 
\textbf{Example 2.} Regard the set of complex $\mathfrak{gl}(1,1)$-RTLSG structures  in $\mathcal V_{\mathfrak{gl}}\backslash\mathcal V_{\mathfrak{sl}}$ associated to the Harish-Chandra superpair given by the standard matrix exponential on the even part in the representation (\ref{mat}). In this case $\Gamma$ is the rank two lattice generated by the elements  $\delta_\pm:=\pi i(A+B)\pm \pi i(A-B)$. Assume that an automorphism $\psi_{x,y,u,v}^+$ of the Lie superalgebra restricts to a group automorphism of $\Gamma$. Then by (\ref{morph}) 
\begin{align*}
\psi_{x,y,u,v}^+(\delta_+)=u\delta_++(u-1)\delta_- \qquad \mbox{ and } \qquad
 \psi_{x,y,u,v}^+(\delta_-)=v\delta_++(v+1)\delta_-
\end{align*}
The restriction is an automorphism of $\Gamma$ whenever $\left(\begin{smallmatrix}u&u-1\\v&v+1\end{smallmatrix}\right) \in GL_2(\mathbb Z)$ hence $u,v \in \mathbb Z$ with $|u+v|=1$. Altogether the continuable symmetries of $\Gamma$ build the  semi-direct product  $\mathfrak S_2^2 \ltimes \mathbb Z$ where the action of $\mathfrak S_2^2$ on $\mathbb Z$ is given by the product of the signs. This group acts effectively on the set of values of the invariants in Theorem \ref{thm:res} \textit{(d)}. Its orbits classify the RTLSGs.
\bibliographystyle{alpha}

\vspace*{1cm}

\parbox[t]{6cm}{
Matthias Kalus\\
Fakult\"at f\"ur Mathematik\\ 
Ruhr-Universit\"at Bochum,\\ 
Universit\"atsstra\ss e 150\\ 
D-44801 Bochum, Germany\\
Matthias.Kalus@rub.de}

\end{document}